\documentclass{article}
\usepackage[utf8]{inputenc}

\usepackage{hyperref, url}               
\hypersetup{bookmarksdepth=subsubsection}
\usepackage[english]{babel}         
\usepackage[utf8]{inputenc}         

\usepackage{csquotes}               
\usepackage{amssymb, amsmath, amsthm}
\usepackage{stmaryrd}
\usepackage{dsfont}                 
\usepackage{tikz-cd, tikz, mathtools}
\usepackage{enumerate}			    
\usepackage[capitalise]{cleveref}   
\usepackage{geometry}
\usepackage[
backend=biber,natbib,style=alphabetic,maxbibnames=99
]{biblatex}
\renewbibmacro{in:}{}               
\newcommand{\notehelper}[3]{\textcolor{#3}{$\blacksquare$}\marginpar{\ifodd\thepage\raggedright\else\raggedleft\fi\color{#3}\tiny \textbf{#2:} #1}}

\DeclareMathOperator{\gl}{gl}
\DeclareMathOperator{\Stk}{Stk}
\DeclareMathOperator{\DiffStk}{DiffStk}
\DeclareMathOperator{\SepStk}{SepStk}
\DeclareMathOperator{\EuclStk}{EuclStk}

\DeclareMathOperator{\Orbfld}{Orbfld}
\DeclareMathOperator{\Mfld}{Mfld}
\DeclareMathOperator{\Eucl}{Eucl}

\DeclareMathOperator{\htp}{htp}
\newcommand{\Piinfty}{\Pi_{\infty}}

\newcommand{\Deltaalg}[1]{\Delta^{#1}_{\mathrm{alg}}}

\DeclareMathOperator{\Open}{Open}
\DeclareMathOperator{\open}{open}

\DeclareMathOperator{\Hom}{Hom}
\DeclareMathOperator{\iHom}{\underline{\operatorname{Hom}}}

\DeclareMathOperator{\Fun}{Fun}

\DeclareMathOperator{\fgt}{fgt}
\DeclareMathOperator{\id}{id}
\DeclareMathOperator{\pr}{pr}

\DeclareMathOperator{\const}{const}

\DeclareMathOperator{\colim}{colim}
\let\lim\relax
\DeclareMathOperator{\lim}{lim}

\newcommand{\ev}{\textup{ev}}

\DeclareMathOperator{\Shv}{Shv}
\DeclareMathOperator{\PSh}{PSh}

\DeclareMathOperator{\pt}{pt}

\newcommand{\catop}{^{\textup{op}}}

\newcommand{\catname}[1]{{\textup{#1}}}

\newcommand{\Spc}{\catname{Spc}}
\newcommand{\Sp}{\catname{Sp}}

\newcommand{\Top}{\catname{Top}}

\newcommand{\Grpd}{\catname{Grpd}}

\newcommand{\Fin}{\catname{Fin}}

\newcommand{\Cat}{\catname{Cat}}
\newcommand{\PrL}{\textup{Pr}^{\textup{L}}}

\newcommand{\Orb}{\textup{Orb}}

\newcommand{\Glo}{\textup{Glo}}

\newcommand{\qin}{\quad\in\quad}

\newcommand{\abs}[1]{\lvert #1 \rvert}

\let\op\relax
\DeclareMathOperator{\op}{op}

\DeclareMathOperator{\N}{\mathbb N}

\DeclareMathOperator{\R}{\mathbb R}

\renewcommand{\phi}{\varphi}
\renewcommand{\epsilon}{\varepsilon}

\newcommand{\Cc}{\mathcal{C}}
\newcommand{\Dd}{\mathcal{D}}

\newcommand{\Ff}{\mathcal{F}}
\newcommand{\Gg}{\mathcal{G}}

\newcommand{\Uu}{\mathcal{U}}

\newcommand{\Xx}{\mathcal{X}}
\newcommand{\Yy}{\mathcal{Y}}

\newcommand{\bH}{\mathbf{H}}

\newcommand{\bbB}{\mathbb{B}}

\newcommand{\bbR}{\mathbb{R}}

\theoremstyle{plain}
\newtheorem{theorem}{Theorem}[section]

\newtheorem{corollary}[theorem]{Corollary}
\newtheorem{lemma}[theorem]{Lemma}
\newtheorem{proposition}[theorem]{Proposition}

\theoremstyle{definition}

\newtheorem{construction}[theorem]{Construction}
\newtheorem{convention}[theorem]{Convention}
\newtheorem{definition}[theorem]{Definition}

\newtheorem{notation}[theorem]{Notation}

\newtheorem{remark}[theorem]{Remark}

\newtheorem*{remark*}{Remark}

\theoremstyle{plain}
\newtheorem*{theorem*}{Main Theorem}
\newtheorem*{corollary*}{Corollary}

\usepackage[capitalise]{cleveref}  

\crefname{construction}{construction}{constructions}
\crefname{convention}{convention}{conventions}
\crefname{corollary}{corollary}{corollaries}
\crefname{definition}{definition}{definitions}
\crefname{lemma}{lemma}{lemmas}
\crefname{notation}{notation}{notations}
\crefname{proposition}{proposition}{propositions}
\crefname{remark}{remark}{remarks}
\crefname{theorem}{theorem}{theorems}

\Crefname{construction}{Construction}{Constructions}
\Crefname{convention}{Convention}{Conventions}
\Crefname{corollary}{Corollary}{Corollaries}
\Crefname{definition}{Definition}{Definitions}
\Crefname{lemma}{Lemma}{Lemmas}
\Crefname{notation}{Notation}{Notations}
\Crefname{proposition}{Proposition}{Propositions}
\Crefname{remark}{Remark}{Remarks}
\Crefname{theorem}{Theorem}{Theorems}

\newcommand{\iso}{\xrightarrow{\;\smash{\raisebox{-0.5ex}{\ensuremath{\scriptstyle\sim}}}\;}}
\newcommand{\isol}{\xleftarrow{\;\smash{\raisebox{-0.5ex}{\ensuremath{\scriptstyle\sim}}}\;}}

\tikzcdset{pullback/.style = {phantom, "\lrcorner"{very near start}}}

\newcommand{\quot}{/\!\!/}

\newcommand{\Addresses}{{
		\bigskip
		\footnotesize
		
		A.~Clough, \textsc{NYU Abu Dhabi, Center for Quantum and Topological Systems, Abu Dhabi, Abu Dhabi, AE}\par\nopagebreak
		\textit{E-mail address:} \texttt{adrian.clough@gmail.com}
		
		\medskip
		
		B.~Cnossen, \textsc{University of Regensburg, Department of Mathematics, Universitätsstraße 31, Regensburg, DE 93053}\par\nopagebreak
		\textit{E-mail address:} \texttt{bastiaan.cnossen@ur.de}
		
		\medskip
		
		S.~Linskens (Corresponding author), \textsc{University of Regensburg, Department of Mathematics, Universitätsstraße 31, Regensburg, DE 93040}\par\nopagebreak
		\textit{E-mail address:} \texttt{sil.linskens@ur.de}
		
}}

\title{Global spaces and the homotopy theory of stacks}
\author{Adrian Clough\thanks{
		The author acknowledges support by \emph{Tamkeen} under \emph{NYUAD Research Institute grant} \texttt{CG008}.
	}, Bastiaan Cnossen, Sil Linskens}

\date{\today}

\begin{document}
	\maketitle

	\begin{abstract}
		We show that the $\infty$-category of global spaces is equivalent to the homotopy localization of the $\infty$-category of sheaves on the site of separated differentiable stacks, following a philosophy proposed by Gepner-Henriques. We further prove that this $\infty$-category of sheaves is a cohesive $\infty$-topos and that it fully faithfully contains the singular-cohesive $\infty$-topos of Sati-Schreiber.
	\end{abstract}
	
	\tableofcontents
	
	\section{Introduction}
	Cohomology theories defined in geometric contexts often admit natural equivariant extensions. For example, if $G$ is a suitable group object acting on an object $X$ (e.g.\ a scheme, complex analytic space, or smooth manifold), an equivariant version of K-theory may be obtained by considering $G$-equivariant vector bundles on $X$. Two features of such theories will be important to us. First, they are \emph{homotopy invariant}: equivariant K-theory, for instance, is unchanged under $G$-homotopies. Second, they are defined uniformly in the group $G$, hence belong to the field of \emph{global homotopy theory} \cite{schwede2018global}; equivalently, the value of such theory on the object $X$ only depends on the quotient stack $X \quot G$. For example, by identifying $G$-equivariant vector bundles on $X$ with vector bundles on the quotient stack $X \quot G$, one obtains an equivalence $K_G(X) \simeq K(X \quot G)$ (\cite{AdemRuan2003Twisted, TXLG_Twisted_Ktheory}; see also \cite{Toen1999RR, VezzosiVistoli2002KTheory}). The study of such theories on stacks is motivated in part by their appearance in mathematical physics, notably in equivariant elliptic cohomology \cite{BE2024Survey}, and in symplectic geometry, through Floer homotopy theory \cite{AB2024Floer}.
	
	In this article we study the \emph{homotopy theory of stacks}, defined as the universal home of such cohomology theories, and relate it to other suggestions for the homotopy theory of stacks in the literature \cite{GepnerHenriques2007Orbispaces,satiSchreiber2020proper, schwede2018global}. We follow the approach of Morel–Voevodsky \cite{MorelVoevodsky1999Motivic}, who obtained a homotopy theory of schemes by localizing a sheaf category on a site of schemes at the projection maps $X \times \mathbb{A}^1 \to X$. In our situation, we replace schemes by \emph{separated differentiable stacks}: sheaves of groupoids $\Xx$ on the site of smooth manifolds which admit a smooth atlas by a manifold and whose diagonal $\Xx \to \Xx \times \Xx$ is representably proper. The resulting $(2,1)$-category $\SepStk$ admits a Grothendieck topology given by open coverings, and we define the homotopy theory of stacks as the localization of the sheaf category of $\SepStk$ at the projections $\Xx \times \R \to \Xx$.

	Previous results in the literature have identified analogous homotopy localizations in topology:
	\begin{itemize}
		\item Let $\Mfld$ denote the category of manifolds, equipped with the open covering topology. Dugger argued in \cite[Prop.~8.3]{Dugger2001Universal} that the localization of the sheaf category $\Shv(\Mfld)$ at the projection maps $\{M \times \R \to M\}_{M \in \Mfld}$ is equivalent to the category $\Spc$ of spaces. The localization functor $\Piinfty\colon \Shv(\Mfld) \to \Spc$ sends a manifold to its singular homotopy type.
		\item For a compact Lie group $G$, let $\Mfld_G$ denote the category of (smooth) $G$-manifolds, equipped with the equivariant open covering topology. The first-named author showed in \cite[Th.~4.4.16]{Cnossen2023PhD} that the localization of $\Shv(\Mfld_G)$ at the projection maps $\{M \times \R \to M\}_{M \in \Mfld}$ is equivalent to the presheaf category $\PSh(\Orb_G)$ of the category $\Orb_G$ of $G$-orbits. In light of Elmendorf's theorem \cite{elmendorf}, this is equivalent to the category $\Spc_G$ of $G$-spaces, and the resulting localization functor $\Piinfty^{G}\colon\Shv(\Mfld_G) \to \Spc_G$ sends a $G$-manifold to its equivariant homotopy type.
	\end{itemize}

	In the case of $\Shv(\SepStk)$, the analogue of the orbit category $\Orb_G$ is the \emph{global indexing category} $\Glo$, introduced by Gepner and Henriques \cite{GepnerHenriques2007Orbispaces} in their study of the homotopy theory of topological stacks.\footnote{We follow the notation and terminology of \cite{rezk2014global}; Gepner and Henriques denoted this category by $\Orb$.} The objects of $\Glo$ are arbitrary compact Lie groups, while the space of morphisms between $\bbB G$ and $\bbB H$ is the homotopy type of the mapping stack between them, i.e.\ $\Piinfty\iHom_{\Shv(\Mfld)}(\bbB G,\bbB H)$. We may think of $\Glo$ as uniformly encoding all the orbit categories $\Orb_G$ simultaneously: sending an orbit $G/H$ to the associated map of classifying stacks $\bbB H \to \bbB G$ induces a fully faithful functor $\Orb_G \hookrightarrow \Glo_{/\bbB G}$. Presheaves on $\Glo$ are called \emph{global spaces}. In the spirit of Elmendorf's theorem, the value at $\bbB G$ may be thought of as the $G$-isotropy space of the global space. The relation between global spaces and equivariant homotopy theory was studied systematically by Rezk in \cite{rezk2014global}.
	
	Our main result shows that the homotopy localization of sheaves on separated stacks is precisely the category of global spaces:

	\begin{theorem*}[\ref{thm:MainTheorem}]
		\label[theorem]{introthm:MainTheorem}
		There exists a global homotopy type functor
		\[
			\Piinfty^{\Glo}\colon\Shv(\SepStk)\to\PSh(\Glo)	
		\]
		which exhibits its target as the localization of its source at the projections $\Xx\times\R\to\Xx$ for $\Xx \in \SepStk$. For a separated stack $\Xx$, its global homotopy type $\Piinfty^{\Glo}(\Xx)$ is given on objects by
		\[
			\Piinfty^{\Glo}(\Xx)(\bbB G)\simeq \Piinfty\iHom_{\Shv(\Mfld)}(\bbB G,\Xx).
		\]
	\end{theorem*}

	By taking coefficients in spectra, the theorem provides an equivalence
	\[
		\Shv^{\htp}(\SepStk;\Sp) \iso \PSh(\Glo;\Sp)
	\]
	between spectrum-valued homotopy invariant sheaves on separated stacks and spectrum-valued presheaves on $\Glo$; in particular, each such presheaf gives rise to a ``naive'' cohomology theory on stacks. It would be interesting and fruitful to extend this comparison to ``genuine'' cohomology theories on stacks (which are as of yet undefined) and \emph{global spectra} in the sense of \cite{schwede2018global}. We hope to return to this question in future work.
	
	We emphasize that the open covering topology on $\SepStk$ is \emph{coarser} than the effective epimorphism topology that is often considered in the literature. To illustrate this, consider the circle $S^1$ equipped with the $C_2$-action given by flipping across an axis. The $C_2$-equivariant open cover of $S^1$ by its two hemispheres induces a covering in $\SepStk$ of the quotient stack $S^1 \quot C_2$ by two copies of $\bbR \quot C_2$. In contrast, the quotient map $S^1 \to S^1 \quot C_2$ is \emph{not} a covering in $\SepStk$. More generally, the open covering topology does not allow any gluings which introduce new isotropy groups. 
	
	The open covering topology was already considered implicitly by Satake and Thurston when they defined orbifolds as generalized manifolds glued from quotient stacks $\bbR^n \quot G$, for linear representations of finite groups (see \cite{Satake1956Orbifold,Thurston1978ThreeManifolds}). We may loosely think of separated stacks as generalized orbifolds in which the isotropy groups may be arbitrary compact Lie groups. Indeed, every separated stack admits an open cover by Euclidean quotient stacks $\bbR^n \quot G$ for linear representations of compact Lie groups (see \Cref{prop:Linearization_Theorem} for a recollection). The full subcategory of $\SepStk$ spanned by separated stacks with finite isotropy groups (or equivalently, those covered by linear representations of finite groups) is equivalent to the $(2,1)$-category $\Orbfld$ of orbifolds.

	The preceding paragraph motivates the following generalization of our main result. Given a collection $\Ff$ of (isomorphism classes of) compact Lie groups, denote by $\SepStk_{\Ff} \subseteq \SepStk$ the full subcategory on those separated stacks $\Xx$ whose isotropy groups are contained in $\Ff$, and denote by $\Glo_{\Ff} \subseteq \Glo$ the full subcategory spanned by the objects $\bbB G$ for $G \in \Ff$.

	\begin{corollary*}[\ref{cor:Main_Result_For_Families}]
		The functor $\Piinfty$ restricts to a functor $\Shv(\SepStk_{\Ff}) \to \PSh(\Glo_{\Ff})$ which exhibits its target as the homotopy localization of its source.
	\end{corollary*}
	
	Taking $\Ff$ to be the class of finite groups, this identifies the homotopy theory of orbifolds with the category of so-called ``$\Fin$-global spaces''.

	For the class of finite groups, Sati and Schreiber \cite[\textsection 3.2]{satiSchreiber2020proper}, \cite[\textsection 3.3]{satiSchreiber2021equivariantPublished} pursue a different approach to combining differential geometry and equivariant homotopy theory. They study the \emph{singular cohesive} $\infty$-topos
	\[
		\bH = \Shv(\Glo_{\Fin}\times\Eucl) \simeq \PSh(\Glo_{\Fin})\otimes \Shv(\Mfld)\simeq \Fun(\Glo_{\Fin}^{\op},\Shv(\Mfld)),
	\]
	where $\Eucl$ denotes the category of smooth manifolds diffeomorphic to $\bbR^n$. This $\infty$-topos is cohesive over both $\PSh(\Glo_{\Fin})$ and $\Shv(\Mfld)$, with several commuting modalities. This feature underlies their synthetic axiomatics of orbifold geometry and is described in their broader program as a key ingredient in an anticipated construction of orbifold/equivariant nonabelian differential cohomology via Chern characters. By contrast, the $\infty$-topos $\Shv(\SepStk_\Ff)$ exhibits a different cohesive behavior (see \cref{subsec:cohesion}), which we expect to play a role in differential orbifold cohomology theory in the style of \cite{bunke2016SheavesOfSpectra} (i.e., using differential hexagons). We provide a comparison between $\bH$ and $\Shv(\SepStk_{\Fin})$ in \Cref{subsec:Cohesion_for_orbifolds}.
	
	\subsection*{Structure of the article}
	We begin in Subsection~\ref{subsec:SepStk} by recalling the definitions of separated stacks and associated notion of sheaves. In Subsection~\ref{subsec:htpinv} we recall the definition of homotopy invariant sheaves on $\SepStk$ and show that the resulting $\infty$-topos $\Shv^{\htp}(\SepStk)$ is equivalent to a presheaf $\infty$-topos on an $\infty$-category which we denote by $\Glo^{\Stk}$. In Subsection \ref{sec:global-spaces} we prove that the $\infty$-category $\Glo^{\Stk}$ is equivalent to the global indexing category $\Glo$ as traditionally defined in the literature, which proves our main theorem. We explain in \ref{subsec:isotropy} how to deduce the analogous result for separated stacks with isotropy groups in a specified family $\Ff$.
	
	In Section~\ref{subsec:cohesion} we then apply the previous results to show that $\Shv(\SepStk_{\Ff})$ is cohesive, in the sense of Schreiber \cite{schreiber2013differential}. In the case where $\Ff$ is a collection of finite groups, we improve the previous result by showing that $\Shv(\Orbfld_{\Ff})$ is already cohesive relative to the base $\infty$-topos $\PSh(\Glo_{\Ff})$. We also relate $\Shv(\Orbfld)$ to the singular-cohesive topos $\bH$ of Sati--Schreiber \cite{satiSchreiber2020proper}.
	
	\subsection*{Acknowledgments}
	We thank the anonymous referee for a careful reading of our article and various helpful suggestions. We further thank Daniele Velati for pointing out an oversight in our original proof of \Cref{prop:Homotopy_Invariant_Presheaf_On_EuclStk_Is_Sheaf}.
	
	\section{The homotopy theory of stacks}\label{sec:SepStk}
	In this section we prove our main theorem by identifying the homotopy localization of $\Shv(\SepStk)$ with the category of global spaces $\PSh(\Glo)$.
	
	\subsection{Separated stacks}\label{subsec:SepStk}
	We establish some of our conventions on notation and terminology regarding stacks.
	
	\begin{notation}
		We let $\Mfld$ denote the category of smooth manifolds, and regard $\Mfld$ as a Grothendieck site by equipping it with the \textit{open covering topology} in which the covering families are precisely the open covers $\{U_i \hookrightarrow M\}_{i \in I}$. We denote by $\Shv(\Mfld;\Grpd)$ the (2,1)-category of sheaves of groupoids with respect to this Grothendieck topology, and refer to its objects as \textit{stacks on $\Mfld$}. Note that the Yoneda embedding induces a fully faithful functor $\Mfld \hookrightarrow \Shv(\Mfld;\Grpd)$; we will abuse notation by identifying a manifold with its associated stack on $\Mfld$.
	\end{notation}
	
	\begin{definition}[{Differentiable stack, \cite[Definition~2.15]{BehrendXu2011Stacks}}]
		A stack $\Xx \in \Shv(\Mfld;\Grpd)$ is called a \textit{differentiable stack} if there exists a smooth manifold $M$ equipped with a morphism $p\colon M \twoheadrightarrow \Xx$ of stacks on $\Mfld$ satisfying the following two properties:
		\begin{enumerate}[(1)]
			\item The morphism $p$ is an \textit{effective epimorphism}, which equivalently means that for any smooth manifold $N$ and any morphism of stacks $f\colon N \to \Xx$ there exists an open cover $\{\iota_i\colon U_i \hookrightarrow N\}_{i \in I}$ of $N$ such that each composite $f \circ \iota_i\colon U_i \to \Xx$ factors through $p\colon M \to \Xx$;
			\item The morphism $p$ is a \textit{representable submersion}, which means that for any smooth manifold $N$ and any morphism of stacks $f\colon N \to \Xx$ the fiber product $M \times_{\Xx} N$ of stacks is (representable by) a smooth manifold, and the projection map $M \times_{\Xx} N \to N$ is a smooth submersion of smooth manifolds.
		\end{enumerate}
		A map $p\colon M \twoheadrightarrow \Xx$ satisfying properties (1) and (2) is also called a \textit{smooth atlas} of $\Xx$. We denote by $\DiffStk \subseteq \Shv(\Mfld;\Grpd)$ the full subcategory spanned by the differentiable stacks.
	\end{definition}

	Clearly every smooth manifold $M$ is a differentiable stack, with atlas given by the identity map $\id_M \colon M \to M$. More interestingly, also every Lie groupoid $\Gg = (\Gg_1 \rightrightarrows \Gg_0)$ gives rise to a differentiable stack $\bbB \Gg$: for a smooth manifold $N$ we define the groupoid $\bbB \Gg(N)$ as the groupoid of principal $\Gg$-bundles on $N$. One can show that the map $\Gg_0 \to \bbB \Gg$ classifying the canonical principal $\Gg$-bundle on $\Gg_0$ is a smooth atlas for $\bbB \Gg$; we refer to \cite[Section~3.2]{GepnerHenriques2007Orbispaces} for a discussion in the topological setting. In particular, every smooth manifold $M$ equipped with a smooth action by a Lie group $G$ gives rise to a \textit{quotient stack} $M\quot G$, obtained by applying the above construction to the action groupoid $G \ltimes M = (G \times M \rightrightarrows M)$. Taking the manifold to be the point equipped with the trivial action, we obtain for every Lie group $G$ a differentiable stack $\bbB G = \pt \quot G$, called the \textit{classifying stack of $G$}, so named because it classifies (i.e.\ represents) the groupoid of principal $G$-bundles.
	
	\begin{definition}[{Representable morphism, \cite[Section~2.3]{BehrendXu2011Stacks}}]
		Let $f\colon \Xx \to \Yy$ be a morphism of differentiable stacks.
		\begin{enumerate}[(1)]
			\item The morphism $f$ is called \textit{representable} if for every smooth atlas $M \twoheadrightarrow \Yy$ the fiber product $M \times_{\Yy} \Xx$ of stacks is (representable by) a smooth manifold $N$. (Note that the map $N \twoheadrightarrow \Xx$ is then a smooth atlas of $\Xx$.)
			\item The morphism $f$ is called an \textit{open embedding} (resp.\ \textit{proper}) if it is representable and the induced smooth map $N \to M$ from (1) is an open embedding (resp.\ a proper map).
		\end{enumerate}
	\end{definition}
	
	The open embeddings lead to a natural Grothendieck topology on $\DiffStk$:
	
	\begin{definition}[Open covering topology]	\label[definition]{def:Open_cover_topology}
		Given a differentiable stack $\Xx$, an \textit{open covering of $\Xx$} is a collection $\{\iota_i\colon \Uu_i \hookrightarrow \Xx\}_{i \in I}$ of open embeddings of differentiable stacks such that the map $(\iota_i) \colon \bigsqcup_{i \in I} \Uu_i \twoheadrightarrow \Xx$ is an effective epimorphism. This determines a Grothendieck topology on $\DiffStk$ in which a sieve on a differentiable stack $\Xx$ is a covering sieve if and only if it contains an open covering of $\Xx$. We denote by $\Shv(\DiffStk)$ the resulting $\infty$-topos of sheaves of spaces on $\DiffStk$, and denote by $L_{\open}\colon \PSh(\DiffStk) \to \Shv(\DiffStk)$ the associated sheafification functor.
	\end{definition}
	
	For our purposes, it is convenient to restrict attention to a smaller class of differentiable stacks:
	
	\begin{definition}[{Separated stack, \cite[Definition~2.23 ]{BehrendXu2011Stacks}}]
		A differentiable stack $\Xx$ is called \textit{separated} if its diagonal $\Xx \to \Xx \times \Xx$ is proper. We denote by $\SepStk \subseteq \DiffStk$ the full subcategory spanned by the separated differentiable stacks. The Grothendieck topology restricts to one on $\SepStk$, resulting in an $\infty$-topos $\Shv(\SepStk)$.
	\end{definition}
	
	It is not difficult to see that every quotient stack $M\quot G$ of a smooth manifold by a smooth action of a \textit{compact} Lie group $G$ is separated. It turns out that every separated stack is in fact \textit{locally} of this form.
	
	\begin{definition}[Euclidean stack]
		A differentiable stack $\Xx$ is called \textit{Euclidean} if it is of the form $V\quot G$, where $G$ is a compact Lie group and $V$ is a finite-dimensional representation of $G$. We let $\EuclStk$ denote the $(2,1)$-category of Euclidean stacks, and equip it with the open covering topology inherited from $\DiffStk$.
	\end{definition}
	
	\begin{proposition}[Linearization theorem]
		\label[proposition]{prop:Linearization_Theorem}
		Every separated stack $\Xx$ is \textit{locally Euclidean}, in the sense that every point $x \in \Xx$ admits an open neighborhood $x \in \Uu \subseteq \Xx$ isomorphic to a Euclidean stack.
	\end{proposition}
	\begin{proof}
		In terms of Lie groupoids, this result is due to \cite{Zung2006Linearization} and \cite{Weinstein2002Linearization}, see also \cite[Corollary~3.9]{CrainicStruchiner2013Linearization} and \cite[Corollary~3.11]{PPT2014OrbitSpaces} for more precise formulations. The translation to the language of separated stacks can be found in \cite[Theorem~3.7.2]{Cnossen2023PhD}.
	\end{proof}
	
	\begin{corollary}\label[corollary]{cor:hypercomplete}
		The $\infty$-topos $\Shv(\SepStk)$ is hypercomplete.
	\end{corollary}
	\begin{proof}
		For a separated stack $\Xx$, let $\Open(\Xx) \subseteq \SepStk_{/\Xx}$ denote the full subcategory spanned by the open embeddings $\Uu \hookrightarrow \Xx$. The forgetful functor $\Open(\Xx) \to \SepStk$ is both a continuous and cocontinuous morphism of sites, and hence induces a restriction functor
		\[
		(-)\vert_{\Xx}\colon \Shv(\SepStk) \to \Shv(\Open(\Xx))
		\]
		that preserves both limits and colimits. Observe that for a sheaf $\Ff \in \Shv(\SepStk)$ the value $\Ff(\Xx)$ of $\Ff$ at $\Xx$ is equivalent to the global sections of the restricted sheaf $\Ff\vert_{\Xx}$, hence we see that the functors $(-)\vert_{\Xx}$ are jointly conservative when $\Xx$ runs over all separated stacks. Because an $\infty$-topos is hypercomplete if and only if $\infty$-connected morphisms are equivalences and the functors $(-)|_{\Xx}$ all preserve $\infty$-connected morphisms, it suffices to show that each of the $\infty$-topoi $\Shv(\Open(\Xx))$ is hypercomplete. By \Cref{prop:Linearization_Theorem}, we may find an open cover $\{\Uu_{\alpha}\}$ of $\Xx$ by Euclidean substacks. For each $\alpha$, the inclusion functors $\Open(\Uu_{\alpha}) \hookrightarrow \Open(\Xx)$ are continuous and cocontinuous morphisms of sites, and so the restriction functors $\Shv(\Open(\Xx)) \to \Shv(\Uu_{\alpha})$ preserve limits and colimits, hence $\infty$-connected maps. As they are jointly conservative, they jointly detect $\infty$-connected maps, and so we may reduce to the case where $\Xx = V\quot G$ is a Euclidean stack itself. In this case, note that there is an inclusion $\Open(V\quot G) \hookrightarrow \Open(V)$ whose image consists of the $G$-invariant open subsets of $V$, and that this inclusion admits a left adjoint sending $U \subseteq V$ to $G \cdot U = \{gu \mid g \in G, u \in U\}$. Restriction along this left adjoint thus induces a fully faithful left-exact colimit-preserving functor $\Shv(\Open(V\quot G)) \hookrightarrow \Shv(\Open(V)) = \Shv(V)$, which in particular preserves and detects $\infty$-connected maps. It thus suffices to show that $\Shv(V)$ is hypercomplete. Since the covering dimension of $V$ is equal to its dimension as a vector space and therefore finite, this is an instance of \cite[Theorem 7.2.3.6, Corollary 7.2.1.12]{lurie2009HTT}.
	\end{proof}
	
	\begin{corollary}
		\label[corollary]{cor:Generated_By_Euclidean_Stacks}
		Restriction along the inclusion $\EuclStk \hookrightarrow \SepStk$ determines an equivalence of $\infty$-categories $\Shv(\SepStk) \iso \Shv(\EuclStk)$.
	\end{corollary}
	\begin{proof}
		By \Cref{prop:Linearization_Theorem}, the subcategory $\EuclStk$ of $\SepStk$ is a \textit{basis for the Grothendieck topology}: for every separated stack $\Xx$ there exists an open cover $\{\Uu_{x} \hookrightarrow \Xx\}$ by Euclidean stacks $\Uu_x$. Since $\Shv(\SepStk)$ is hypercomplete by \Cref{cor:hypercomplete}, the result is now an instance of \cite[Corollary~3.12.13]{BarwickHaine2020exodromy} or \cite[Appendix~A]{AokiTensor2023}.
	\end{proof}
	
	The $\infty$-topos $\Shv(\SepStk)$ is closely related to $\Shv(\Mfld)$:
	
	\begin{proposition}
		\label[proposition]{prop:Adjunctions_SepStk_and_Mfld}
		The inclusion functor $w\colon \Mfld \hookrightarrow \SepStk$ is both a continuous and cocontinuous morphism of Grothendieck sites. In particular, it induces a pair of adjunctions
		\[
		\begin{tikzcd}[column sep = large]
			\Shv(\Mfld) 
			\ar[hookrightarrow, shift left=8, "w^*"{description}]{rrr} 
			\ar[rrr,phantom,shift left=4,"{\scalebox{.6}{$\bot$}}"{description}] 
			\ar["w_*"{description}, leftarrow]{rrr}
			\ar[rrr,phantom,shift right=4,"{\scalebox{.6}{$\bot$}}"{description}] 
			\ar[hookrightarrow,shift right=8, "w^!"{description}]{rrr}
			&&& 
			\Shv(\SepStk),
		\end{tikzcd}
		\]
		where the functors $w^*$ and $w^!$ are fully faithful and $w^*$ is left exact.
	\end{proposition}
	\begin{proof}
		The first part of the proposition is clear: the inclusion preserves pullbacks along open embeddings, and the covering families are defined on both sides by open covers. The second part is an immediate consequence.
	\end{proof}
	
	\begin{remark}
		\label[remark]{rmk:Embeddings_Of_SepStk}
		Letting $y\colon \SepStk \hookrightarrow \Shv(\SepStk)$ denote the Yoneda embedding, it is immediate from the definitions that the composite
		\[
		\SepStk \xhookrightarrow{y} \Shv(\SepStk) \xrightarrow{w_*} \Shv(\Mfld)
		\]
		is equivalent to the canonical inclusion $\SepStk \hookrightarrow \DiffStk \hookrightarrow \Shv(\Mfld)$. Using this, one can in turn show that the composite
		\[
		\SepStk \hookrightarrow \Shv(\Mfld) \xhookrightarrow{w^!} \Shv(\SepStk)
		\]
		is equivalent to the Yoneda embedding. As a result, the two canonical embeddings of $\SepStk$ into $\Shv(\Mfld)$ and $\Shv(\SepStk)$ are identified with each other under the functors $w_*$ and $w^!$.
	\end{remark}
	
	\subsection{Homotopy invariant sheaves}\label{subsec:htpinv}
	We will now introduce the $\infty$-category $\Shv^{\htp}(\SepStk)$ of \textit{homotopy invariant} sheaves, and show that it is equivalent to a presheaf category.
	
	\begin{definition}[Homotopy invariance]
		A presheaf $\Ff$ on $\SepStk$ is called \textit{homotopy invariant} if for every separated stack $\Xx$ the map $\Ff(\Xx) \to \Ff(\Xx \times \R)$, induced by the projection map $\Xx \times \R \to \Xx$, is an equivalence. We let $\PSh^{\htp}(\SepStk)\subseteq \PSh(\SepStk)$ and $\Shv^{\htp}(\SepStk) \subseteq \Shv(\SepStk)$ denote the full subcategories spanned by the homotopy invariant (pre)sheaves. We may similarly speak of homotopy invariant (pre)sheaves on various subsites of $\SepStk$, like $\EuclStk$ and $\Mfld$.
		
		Since the homotopy invariant (pre)sheaves are precisely those objects which are local with respect to the small class of morphisms $\Xx \times \R \to \Xx$, the inclusion functors admit left adjoints, which we denote by
		\[
		L_{\R}\colon \PSh(\SepStk) \rightarrow \PSh^{\htp}(\SepStk) \quad \text{and} \quad L_{\htp}\colon \Shv(\SepStk) \to \Shv^{\htp}(\SepStk)
		\]
		and refer to as the \textit{homotopy localization functors}.
	\end{definition}
	
	\begin{remark}
		\label[remark]{rmk:Test_Homotopy_Invariance_On_EuclStk}
	It follows from \Cref{cor:Generated_By_Euclidean_Stacks} that $\Ff$ is homotopy invariant if and only if its restriction to $\EuclStk$ is homotopy invariant. In particular, the equivalence $\Shv(\SepStk) \iso \Shv(\EuclStk)$ from \Cref{cor:Generated_By_Euclidean_Stacks} restricts to an equivalence $\Shv^{\htp}(\SepStk) \iso \Shv^{\htp}(\EuclStk)$.
	\end{remark}
	
	The homotopy localization functor $L_{\R}$ on presheaves admits an explicit formula known as the \textit{Morel-Voevodsky construction}:
	
	\begin{definition}
		\label[definition]{def:Algebraic_Simplices}
		For a natural number $n \in \N$, define the \textit{$n$-th algebraic simplex} $\Deltaalg{n}$ as the hyperplane in $\R^{n+1}$ given by
		\[
		\Deltaalg{n} := \{(x_0, x_1, \dots, x_n) \in \R^{(n+1)} \mid \sum_{i=0}^n x_i = 1\} \qin \Mfld.
		\]
		Observe that these manifolds assemble into a functor $\Deltaalg{\bullet}\colon \Delta \to \Mfld$.
	\end{definition}
	
	\begin{proposition}
		\label[proposition]{prop:Homotopy_Localization_Is_Morel_Voevodsky}
		For a presheaf $\Ff \in \PSh(\SepStk)$ and a separated stack $\Xx \in \SepStk$ there is a natural equivalence
		\[
		L_{\R}(\Ff)(\Xx) \simeq \colim_{[n] \in \Delta\catop} \Ff(\Xx \times \Deltaalg{n}).
		\]
	\end{proposition}
	\begin{proof}
		The proof is identical to that of \cite[Proposition~5.1.2]{ADH2021differential} who prove the same claim for presheaves on $\Mfld$; here we merely record the overall structure of the proof. If we denote the right-hand side of the equation in the proposition by $H(\Ff)(\Xx)$, this expression defines a functor $H\colon \PSh(\SepStk) \to \PSh(\SepStk)$. It comes equipped with a natural transformation $\eta \colon \id \to H$ given by the canonical maps from $\Ff(\Xx) = \Ff(\Xx \times \Deltaalg{0})$ to $\colim_{[n] \in \Delta\catop} \Ff(\Xx \times \Deltaalg{n})$. One can show that $H(\Ff)$ is always a homotopy invariant presheaf and that the map $\eta_{\Ff}\colon \Ff \to H(\Ff)$ is an equivalence whenever $\Ff$ is homotopy invariant, see \cite[Lemmas~5.3.4 and 5.3.6]{ADH2021differential}. In particular, the image of $H$ consists precisely of the homotopy invariant presheaves. It follows that $\eta_{H\Ff}\colon H(\Ff) \to H(H(\Ff))$ is an equivalence, and one shows that also the map $H(\eta_{\Ff})\colon H(\Ff) \to H(H(\Ff))$ is an equivalence, see \cite[Lemma~5.3.8]{ADH2021differential}. By \cite[Proposition~5.2.7.4]{lurie2009HTT}, it follows that $\eta\colon \id \to H$ exhibits $H$ as a Bousfield localization onto the subcategory of homotopy invariant presheaves, hence must agree with the functor $L_{\R}$.
	\end{proof}
	
	\begin{proposition}
		\label[proposition]{prop:w^*_Preserves_Constant_Sheaves}
		The functor $w^!\colon \Shv(\Mfld) \hookrightarrow \Shv(\SepStk)$ preserves homotopy invariant sheaves.
	\end{proposition}
	\begin{proof}
		Given a homotopy invariant sheaf $\Xx$ on $\Mfld$ we have to show that $w^!\Xx \in \Shv(\SepStk)$ is homotopy invariant, i.e.\ that for every $\Yy \in \SepStk$ the map $(w^!\Xx)(\Yy) \to (w^!\Xx)(\Yy \times \R)$ is an equivalence. By adjunction, this map is the same as the map
		\[
		\Hom_{\Shv(\Mfld)}(\Yy,\Xx) \to \Hom_{\Shv(\Mfld)}(\Yy \times \R,\Xx)
		\]
		induced by the projection map $\Yy \times \R \to \Yy$ in $\Shv(\Mfld)$. Since the collection of $\Yy \in \Shv(\Mfld)$ for which this is true is closed under colimits and contains all smooth manifolds by assumption on $\Xx$, the claim follows.
	\end{proof}
	
	\begin{corollary}
		\label[corollary]{cor:BG_Homotopy_Invariant}
		Let $G$ be a finite group. Then the representable object $\bbB G \in \Shv(\SepStk)$ is homotopy invariant.
	\end{corollary}
	\begin{proof}
		As discussed in \Cref{rmk:Embeddings_Of_SepStk}, the representable sheaf $\bbB G \in \SepStk \subseteq \Shv(\SepStk)$ is the image under $w^!\colon \Shv(\Mfld) \hookrightarrow \Shv(\SepStk)$ of the classifying stack $\bbB G \in \Shv(\Mfld)$. By \Cref{prop:w^*_Preserves_Constant_Sheaves} it thus suffices to show that the latter is homotopy invariant. By definition, $\bbB G$ is a simplicial colimit of the finite sets $G^n$, which is a simplicial diagram contained in the essential image of the constant sheaf functor $\Gamma^*\colon \Spc \hookrightarrow \Shv(\Mfld)$. As this functor preserves colimits, it follows that also $\bbB G$ lies in the essential image of $\Gamma^*$ and thus it is homotopy invariant by \cite[Proposition~4.3.1]{ADH2021differential}.
	\end{proof}
	
	\begin{remark}
	The previous corollary is also a special instance of \cite[Theorem 1.1]{BdBP2019classifying}.
	\end{remark}
	
	\begin{proposition}
		\label[proposition]{prop:OrbitsAreConservative}
		The evaluation functors $\ev_{\bbB G} \colon \Shv^{\htp}(\SepStk) \to \Spc$ for compact Lie groups $G$ are jointly conservative.
	\end{proposition}
	\begin{proof}
		By \Cref{rmk:Test_Homotopy_Invariance_On_EuclStk}, we see that the functors $\ev_{V\quot G}\colon \Shv^{\htp}(\SepStk) \to \Spc$ are jointly conservative if $G$ runs over all compact Lie groups $G$ and $V$ runs over all finite-dimensional $G$-representations. Since the map $V\quot G \to *\quot G = \bbB G$ is a homotopy equivalence, the induced map $\Ff(\bbB G) \to \Ff(V\quot G)$ is an equivalence for any homotopy invariant sheaf $\Ff$ on $\SepStk$, and hence the functor $\ev_{V\quot G}$ is equivalent to $\ev_{\bbB G}$ on $\Shv^{\htp}(\SepStk)$. This finishes the proof.
	\end{proof}
	
	\begin{definition}[Isotropy functors]
		\label[definition]{def:Isotropy_Functors}
		For a compact Lie group $G$, consider the functor $\bbB G \times - \colon \Mfld \to \SepStk$ from smooth manifolds to separated stacks. We suggestively denote by
		\[
		(-)^G\colon \PSh(\SepStk) \to \PSh(\Mfld)
		\]
		the functor given by precomposition with this functor, and refer to it as the \textit{$G$-isotropy functor}.
	\end{definition}
	
	The functor $(-)^G$ is fully compatible with sheafification and homotopy localization:
	
	\begin{proposition}
		\label[proposition]{prop:HFixedPointsPreservesEverything}
		Let $G$ be a compact Lie group.
		\begin{enumerate}[(1)]
			\item The functor $(-)^G\colon \PSh(\SepStk) \to \PSh(\Mfld)$ preserves sheaves.
			\item The functor $(-)^G$ commutes with sheafification:
			\[
			\begin{tikzcd}
				\PSh(\SepStk) \rar{(-)^G} \dar[swap]{L_{\open}} & \PSh(\Mfld) \dar{L_{\open}} \\
				\Shv(\SepStk) \rar{(-)^G} & \Shv(\Mfld).
			\end{tikzcd}
			\]
			\item The functor $(-)^G\colon \PSh(\SepStk) \to \PSh(\Mfld)$ preserves homotopy invariant presheaves.
			\item The functor $(-)^G$ commutes with homotopy localization:
			\[
			\begin{tikzcd}
				\PSh(\SepStk) \rar{(-)^G} \dar[swap]{L_{\R}} & \PSh(\Mfld) \dar{L_{\R}} \\
				\PSh^{\htp}(\SepStk) \rar{(-)^G} & \PSh^{\htp}(\Mfld).
			\end{tikzcd}
			\]
			\item The functor $(-)^G\colon \PSh(\SepStk) \to \PSh(\Mfld)$ restricts to functors
			\begin{align*}
				(-)^G\colon \Shv(\SepStk) &\to \Shv(\Mfld) \\
				(-)^G\colon \Shv^{\htp}(\SepStk) &\to \Shv^{\htp}(\Mfld)
			\end{align*}
			which both preserve limits and colimits.
		\end{enumerate}
	\end{proposition}
	\begin{proof}
		Part (1) is clear from the fact that the functor $\bbB G \times -\colon \Mfld \to \SepStk$ preserves pullbacks along open embeddings and that it sends open covers of smooth manifolds to open covers of stacks.
		
		For part (2), it suffices to show that the functor $\bbB G \times -\colon \Mfld \to \SepStk$ is a cocontinuous functor of sites: given a manifold $M \in \Mfld$ and an open cover $\{\Uu_i\}$ of the stack $\bbB G \times M$, the sieve on $M$ consisting of all maps $N \to M$ such that $\bbB G \times N \to \bbB G \times M$ factors through one of the $\Uu_i \hookrightarrow M$ is a covering sieve. But this is clear, since all open substacks of $\bbB G \times M$ are of the form $\bbB G \times U$ for some open $U \subseteq M$, so $\{\Uu_i\} = \{\bbB G \times U_i\}$ for an open cover $\{U_i\}$ of $M$.
		
		Part (3) is obvious from the definition of homotopy invariance. Part (4) is immediate from the explicit description of the homotopy localization functor $L_{\R}$ given in \Cref{prop:Homotopy_Localization_Is_Morel_Voevodsky} and its analogue for presheaves on $\Mfld$ proved in \cite[Proposition~5.1.2]{ADH2021differential}.
		
		For part (5), the $G$-fixed point functor restricts to sheaves by part (1) and to homotopy invariant sheaves by part (3). At the level of presheaves, the $G$-fixed point functor $(-)^G\colon \PSh(\SepStk) \to \PSh(\Mfld)$ preserves limits and colimits as these are computed pointwise. It follows at once that also its restriction to both $\Shv(\SepStk)$ and $\Shv^{\htp}(\SepStk)$ preserve limits, as the inclusion functors into presheaves preserve limits. Since colimits in (homotopy invariant) sheaves are computed by first computing the colimit in presheaves and then reflecting back into (homotopy invariant) sheaves, the restricted functors also preserve colimits by parts (2) and (4). This finishes the proof.
	\end{proof}
	
	\begin{corollary}
		\label[corollary]{cor:EvaluationOnOrbitsColimitPreserving}
		For every compact Lie group $G$, the evaluation functors
		\begin{align*}
			\ev_{\bbB G} \colon \Shv(\SepStk) &\to \Spc \\
			\ev_{\bbB G} \colon \Shv^{\htp}(\SepStk) &\to \Spc
		\end{align*}
		preserve colimits.
	\end{corollary}
	\begin{proof}
		These evaluation functors can be written as the composites
		\begin{align*}
			\Shv(\SepStk) \xrightarrow{(-)^G} \Shv(\Mfld) \xrightarrow{\Gamma_*} \Spc; \\
			\Shv^{\htp}(\SepStk) \xrightarrow{(-)^G} \Shv^{\htp}(\Mfld) \xrightarrow{\Gamma_*} \Spc,
		\end{align*}
		where $\Gamma_*$ denotes the global section functor. In both composites, the first functor preserves colimits by part (5) of \Cref{prop:HFixedPointsPreservesEverything}. The statement now follows since the functor $\Gamma_*\colon \Shv(\Mfld) \to \Spc$ preserves colimits by \cite[Corollary~4.1.5]{ADH2021differential}, while the functor $\Gamma_*\colon \Shv^{\htp}(\Mfld) \to \Spc$ is an equivalence by \cite[Proposition~4.3.1]{ADH2021differential}.
	\end{proof}
	
	\begin{definition}
		\label[definition]{def:Stacky_Global_Indexing_Category}
		We define the \textit{stacky global indexing category} $\Glo^{\Stk} \subseteq \Shv^{\htp}(\SepStk)$ as the full subcategory spanned by the objects of the form $L_{\htp}(\bbB G)$ for compact Lie groups $G$.
	\end{definition}
	
	We will show in \Cref{sec:global-spaces} below that the stacky global indexing category is equivalent to the usual global indexing category $\Glo$ considered in global equivariant homotopy theory.

	Observe that the hom functor $\Hom(-,-)\colon (\Glo^{\Stk})\catop \times \Shv^{\htp}(\SepStk) \to \Spc$ produces by currying a functor from $\Shv^{\htp}(\SepStk)$ into the presheaf category $\PSh(\Glo^{\Stk}) = \Fun((\Glo^{\Stk})\catop, \Spc)$. The main result of this section is the following:
	
	\begin{theorem}
		\label[theorem]{thm:Homotopy_Invariant_Sheaves_Form_Presheaf_Category}
		The functor $\Shv^{\htp}(\SepStk) \to \PSh(\Glo^{\Stk})$ is an equivalence.
	\end{theorem}
	\begin{proof}
        By adjunction and the Yoneda lemma, we have for every compact Lie group an equivalence
		\[
		\Hom_{\Shv^{\htp}(\SepStk)}(L_{\htp}(\bbB G),-) \simeq \ev_{\bbB G} \colon \Shv^{\htp}(\SepStk) \to \Spc.
		\]
        It then follows from \Cref{prop:OrbitsAreConservative} and \Cref{cor:EvaluationOnOrbitsColimitPreserving} that these functors preserve colimits and are jointly conservative. Therefore the result follows from the generalized Elmendorf's theorem, see \cite[Theorem~3.39]{linskensNardinPol2022Global}.
	\end{proof}
	
	The method of proof of this theorem also allows us to deduce that homotopy invariant presheaves on $\EuclStk$ are automatically sheaves. This will require the following auxiliary lemma:
	
	\begin{lemma}
		\label[lemma]{lem:Comparison_LR_Lhtp_On_EuclStk}
		If $\Ff \in \Shv(\SepStk)$ is a sheaf, then the canonical comparison map $L_{\R}(\Ff) \to L_{\htp}(\Ff)$ restricts to an equivalence $\iota^* L_{\R}(\Ff) \simeq \iota^* L_{\htp}(\Ff)$ of presheaves on $\EuclStk$.
	\end{lemma}
	\begin{proof}
		The evaluations $\ev_{\bbB G} \colon \PSh^{\htp}(\EuclStk) \to \Spc$ for varying compact Lie group $G$ are jointly conservative and both $\iota^* L_{\R}(\Ff)$ and $\iota^* L_{\htp}(\Ff)$ are homotopy invariant. One reduces to showing that the map $L_{\R}(\Ff)(\bbB G) \to L_{\htp}(\Ff)(\bbB G)$ is an equivalence for every such $G$. Via the natural isomorphisms
		\[
		L_{\R}(\Ff)(\bbB G) = L_{\R}(\Ff)^G(\ast) \overset{\ref{prop:HFixedPointsPreservesEverything}}{\simeq} L_{\R}(\Ff^G)(\ast)
		\]
		and
		\[
		L_{\htp}(\Ff)(\bbB G) = L_{\htp}(\Ff)^G(\ast) \overset{\ref{prop:HFixedPointsPreservesEverything}}{\simeq} L_{\htp}(\Ff^G)(\ast),
		\]
		this map identifies with the comparison map $L_{\R}(\Ff^G)(\ast) \to L_{\htp}(\Ff^G)(\ast)$. It thus remains to show that for a sheaf $\Ff' \in \Shv(\Mfld)$, the comparison map $L_{\R}(\Ff') \to L_{\htp}(\Ff')$ is an equivalence when restricted to $\Eucl \subseteq \Mfld$. This is a consequence of the fact that any homotopy invariant presheaf on $\Eucl$ is a sheaf; see e.g.\ \cite[Proposition~4.3.12]{ADH2021differential}.
	\end{proof}
	
	\begin{proposition}
		\label[proposition]{prop:Homotopy_Invariant_Presheaf_On_EuclStk_Is_Sheaf}
		Every homotopy invariant presheaf on $\EuclStk$ is automatically a sheaf. In short, $\Shv^{\htp}(\EuclStk) = \PSh^{\htp}(\EuclStk)$.
	\end{proposition}
	\begin{proof}
        In light of the equivalence from \Cref{cor:Generated_By_Euclidean_Stacks}, it follows from \Cref{thm:Homotopy_Invariant_Sheaves_Form_Presheaf_Category} that the evaluation functors $\ev_{\bbB G} \colon \Shv^{\htp}(\EuclStk) \to \Spc$ induce an equivalence $\Shv^{\htp}(\EuclStk) \iso \PSh(\Glo^{\Stk})$. By 2-out-of-3, it then remains to show that also $\PSh^{\htp}(\EuclStk) \to \PSh(\Glo^{\Stk})$ is an equivalence. The argument is identical to that of \Cref{thm:Homotopy_Invariant_Sheaves_Form_Presheaf_Category}, after observing the following chain of equivalences for every homotopy invariant sheaf $\Ff$:
        \begin{align*}
            \ev_{\bbB G} \Ff &\simeq \Hom_{\PSh(\EuclStk)}(\iota^*(\bbB G),\Ff) \\
            &\simeq \Hom_{\PSh^{\htp}(\EuclStk)}(\iota^* L_{\R}(\bbB G),\Ff) \\
            &\overset{\ref{lem:Comparison_LR_Lhtp_On_EuclStk}}{\simeq} \Hom_{\PSh^{\htp}(\EuclStk)}(\iota^* L_{\htp}(\bbB G),\Ff).\qedhere
        \end{align*}
	\end{proof}
	
	\subsection{Global spaces}\label{sec:global-spaces}
	In this section we recall the definition of the \textit{global indexing category} $\Glo$, and show that it is equivalent to the stacky global indexing category $\Glo^{\Stk}$ from \Cref{def:Stacky_Global_Indexing_Category}. As a consequence, we obtain an equivalence $\Shv^{\htp}(\SepStk) \simeq \Glo\Spc$ between the $\infty$-category of homotopy invariant sheaves on $\SepStk$ and the $\infty$-category of global spaces.
	
	While there exist various models for $\Glo$ in the literature, it will be most convenient for us to work with the model considered by Gepner and Meier in \cite[Definition~2.7]{gepnerMeier2020equivariant}. By \cite[Appendix B]{gepnerMeier2020equivariant}, their model agrees with the other models of $\Glo$ in the literature. We start by recalling the following general construction from \cite{gepnerMeier2020equivariant}:
	
	\begin{construction}
		\label[construction]{cons:Gepner_Meier_Construction}
		Let $\Cc$ be an $\infty$-category admitting finite products and let $\Delta^{\bullet}\colon \Delta \to \Cc$ be a cosimplicial object. Then \cite[Construction~A.6]{gepnerMeier2020equivariant} construct an $\infty$-category $\Cc_{\abs{\Delta}}$ which has the same objects as $\Cc$ but whose mapping spaces are given by
		\[
		\Hom_{\Cc_{\abs{\Delta}}}(X,Y) = \colim_{[n] \in \Delta\catop}\Hom_{\Cc}(X \times \Delta^n, Y).
		\]
		We may construct this $\infty$-category as the colimit $\Cc_{\abs{\Delta}} := \colim_{[n] \in \Delta\catop} \Cc[\Delta^n]$ in $\Cat_{\infty}$, where $\Cc[\Delta^n] \subseteq \Cc_{/\Delta^n}$ is the full subcategory spanned by objects of the form $\pr_2\colon X \times \Delta^n \to \Delta^n$. Since the mapping space in $\Cc[\Delta^n]$ between $X \times \Delta^n$ and $Y \times \Delta^n$ is given by $\Hom_{\Cc}(X \times \Delta^n, Y)$, the results of \cite[Appendix~A]{gepnerMeier2020equivariant} provide the above formula for the mapping spaces of $\Cc_{\abs{\Delta}}$.
		
		If the object $\Delta^0$ is terminal in $\Cc$, then we have $\Cc[\Delta^0] \simeq \Cc$, and hence we obtain a functor
		\[
		L_{\abs{\Delta}}\colon \Cc \to \Cc_{\abs{\Delta}}.
		\]
		The construction of this functor is functorial in the pair $(\Cc,\Delta^{\bullet})$: if $F\colon \Cc \to \Dd$ is a finite-product-preserving functor, then equipping $\Dd$ with the composite cosimplicial object $\Delta \to \Cc \xrightarrow{F} \Dd$ we obtain a functor $F_{\abs{\Delta}}\colon \Cc_{\abs{\Delta}} \to \Dd_{\abs{\Delta}}$ fitting in the following commutative diagram:
		\[
		\begin{tikzcd}
			\Cc \dar[swap]{L_{\abs{\Delta}}} \rar{F} & \Dd \dar{L_{\abs{\Delta}}} \\
			\Cc_{\abs{\Delta}} \rar[dashed]{F_{\abs{\Delta}}} & \Dd_{\abs{\Delta}}.
		\end{tikzcd}
		\]
		Informally, the functor $F_{\abs{\Delta}}$ is given on objects by $F$ and on morphisms by taking the geometric realization of the simplicial map
		\[
		\Hom_{\Cc}(X \times \Delta^n, Y) \to \Hom_{\Dd}(F(X \times \Delta^n), F(Y)) \simeq \Hom_{\Dd}(F(X) \times \Delta^n, F(Y)).
		\]
		Observe that $F_{\abs{\Delta}}$ is fully faithful whenever $F$ is fully faithful.
	\end{construction}
	
	\begin{convention}
		In the case of $\Cc = \Mfld$, we will always choose the cosimplicial object $\Deltaalg{\bullet}\colon \Delta \to \Mfld$ given by the algebraic simplices from \Cref{def:Algebraic_Simplices}. Using the inclusions $\Mfld \hookrightarrow \SepStk \hookrightarrow \Shv(\Mfld) \hookrightarrow \Shv(\SepStk)$ we obtain cosimplicial objects in each of these $\infty$-categories. For $\Cc = \Top$, the category of topological spaces, we will also choose $\Deltaalg{\bullet}$ for convenience, but note that this agrees up to cosimplicial homotopy with the usual functor $\Delta^{\bullet}\colon \Delta \to \Top$.
	\end{convention}
	
	\begin{definition}[Global spaces, {\cite[Definitions~2.7 and 2.8]{gepnerMeier2020equivariant}}]
		We denote by $\Glo$ the full subcategory of $\Shv(\Top)_{\abs{\Delta}}$ spanned by the objects $L_{\abs{\Delta}}(\bbB G)$ for all compact Lie groups $G$, and refer to it as the \textit{(topological) global indexing category}.\footnote{The $\infty$-category $\Glo$ is sometimes also denoted by $\Orb$.} A \textit{global space} is a presheaf $\Glo\catop \to \Spc$, and we let $\Glo\Spc$ denote the $\infty$-category of all presheaves on $\Glo$.
		
		Similarly, we let $\Glo^{\Mfld} \subseteq \Shv(\Mfld)_{\abs{\Delta}} \subseteq \Shv(\SepStk)_{\abs{\Delta}}$ denote the full subcategory spanned by the objects $B_{\gl}G \coloneqq L_{\abs{\Delta}}(\bbB G)$ for all compact Lie groups $G$, and refer to it as the \textit{smooth global indexing category}.
	\end{definition}
	
	\begin{proposition}
		\label[proposition]{prop:Glo_Mfld_Vs_Glo_Top}
		The forgetful functor $\Mfld \to \Top$ induces an equivalence of $\infty$-categories $\Glo^{\Mfld} \iso \Glo$.
	\end{proposition}
	\begin{proof}
		The forgetful functor $\Mfld \to \Top$ preserves open covers, and thus induces a geometric morphism $\Shv(\Top) \to \Shv(\Mfld)$ of $\infty$-topoi. The left adjoint $\Shv(\Mfld) \to \Shv(\Top)$ is the unique colimit-preserving functor which makes the following diagram commute:
		\[
		\begin{tikzcd}
			\Mfld \rar \dar[hookrightarrow] & \Top \dar[hookrightarrow] \\
			\Shv(\Mfld) \rar[dashed] & \Shv(\Top).
		\end{tikzcd}
		\]
		In particular, it preserves the preferred cosimplicial objects $\Deltaalg{\bullet}$, and since it preserves finite products it induces a functor $\Shv(\Mfld)_{\abs{\Delta}} \to \Shv(\Top)_{\abs{\Delta}}$. Since the functor $\Shv(\Mfld) \to \Shv(\Top)$ preserves the classifying stacks of compact Lie groups, it restricts to a functor
		\[
		\Glo^{\Mfld} \to \Glo.
		\]
		We will show that this functor is an equivalence. The functor is clearly essentially surjective, so it remains to show it is fully faithful. By the description of the mapping spaces on both sides, this means we have to show that for all compact Lie groups $H$ and $G$ the maps
		\[
		\Hom_{\Shv(\Mfld)}(\Deltaalg{n} \times \bbB G, \bbB H) \to \Hom_{\Shv(\Top)}(\Deltaalg{n} \times \bbB G, \bbB H)
		\]
		induce an equivalence on geometric realizations. Both of these geometric realizations may be computed as geometric realizations of topological groupoids: in the first case we take the action groupoid of the topological $H$-space $\Hom_{\mathrm{LieGrp}}(G,H)$ of Lie group morphisms $G \to H$ and in the second case of the topological $H$-space $\Hom_{\mathrm{TopGrp}}(G,H)$ of continuous group morphisms $G \to H$. But these agree as any continuous group homomorphism between Lie groups is already smooth, see e.g.\ \cite[Proposition~3.12]{BroeckerTomDieck1985Representations}.
	\end{proof}
	
	\begin{proposition}
		\label[proposition]{prop:Glo_Mfld_Vs_Glo_Stk}
		There is an equivalence of $\infty$-categories $\Glo^{\Mfld} \iso \Glo^{\Stk}$.
	\end{proposition}
	\begin{proof}
		Consider the cosimplicial object $L_{\htp}(\Deltaalg{\bullet}) \colon \Delta \to \Shv^{\htp}(\SepStk)$, giving rise to an $\infty$-category $\Shv^{\htp}(\SepStk)_{\abs{\Delta}}$ via \Cref{cons:Gepner_Meier_Construction}. Let $\Glo' \subseteq \Shv^{\htp}(\SepStk)_{\abs{\Delta}}$ denote the full subcategory spanned by the objects of the form $L_{\abs{\Delta}}L_{\htp}\bbB G$ for all compact Lie groups $G$. Consider now the two functors
		\[
		\Shv^{\htp}(\SepStk) \xrightarrow{\;\;L_{\abs{\Delta}}\;\;} \Shv^{\htp}(\SepStk)_{\abs{\Delta}} \xleftarrow{\;\;(L_{\htp})_{\abs{\Delta}}\;\;} \Shv(\SepStk)_{\abs{\Delta}},
		\]
		where the second functor is induced by the localization functor $L_{\htp}\colon \Shv(\SepStk) \to \Shv^{\htp}(\SepStk)$. We claim that these two functors restrict to equivalences
		\[
		\Glo^{\Stk} \xrightarrow[\sim]{\;\;\;L_{\abs{\Delta}}\;\;\;} \Glo' \xleftarrow[\sim]{\;\;(L_{\htp})_{\abs{\Delta}}\;\;} \Glo^{\Mfld}.
		\]
		It is immediate that both functors are essentially surjective, so we must show that they are fully faithful. For $L_{\abs{\Delta}}$, we need to show that for compact Lie groups $H$ and $G$, the canonical functor
		\[
		\Hom_{\Shv^{\htp}(\SepStk)}(L_{\htp}\bbB H, L_{\htp} \bbB G) \to \colim_{[n] \in \Delta\catop} \Hom_{\Shv^{\htp}(\SepStk)}(L_{\htp}(\bbB H \times \Deltaalg{n}), L_{\htp} \bbB G)
		\]
		is an equivalence. But by adjunction and homotopy invariance of $L_{\htp}(\bbB G)$, we see that the colimit on the right-hand side is a simplicial colimit of a constant simplicial diagram, so that the claim follows. For $(L_{\htp})_{\abs{\Delta}}$, we must show that for compact Lie groups $H$ and $G$, the maps
		\[
		\Hom_{\Shv(\SepStk)}(\bbB H \times \Deltaalg{n}, \bbB G) \xrightarrow{L_{\htp}} \Hom_{\Shv^{\htp}(\SepStk)}(L_{\htp}(\bbB H \times \Deltaalg{n}), L_{\htp}(\bbB G))
		\]
		induce an equivalence on geometric realizations. By \Cref{lem:Comparison_LR_Lhtp_On_EuclStk}, we may identify the geometric realization of the left-hand side with $\Hom_{\Shv(\SepStk)}(\bbB H, L_{\htp}\bbB G)$ and that of the right-hand side with $\Hom_{\Shv(\SepStk)}(\bbB H, L_{\htp}(L_{\htp}\bbB G))$, with the comparison map being given by postcomposition with the unit map $L_{\htp}(\eta)\colon L_{\htp}(\bbB G) \to L_{\htp}(L_{\htp}(\bbB G))$. Since $L_{\htp}$ is a Bousfield localization, this map is an equivalence, finishing the proof.
	\end{proof}
	
	Combining all results established so far, we obtain our main result:
	\begin{theorem}
		\label[theorem]{thm:MainTheorem}
		There exists an equivalence of $\infty$-categories $\Shv^{\htp}(\SepStk) \iso \Glo\Spc$.
	\end{theorem}
	\begin{proof}
		We showed in \Cref{thm:Homotopy_Invariant_Sheaves_Form_Presheaf_Category} that restricting the Yoneda embedding along the inclusion $\Glo^{\Stk} \hookrightarrow \Shv^{\htp}(\SepStk)$ induces an equivalence $\Shv^{\htp}(\SepStk) \iso \PSh(\Glo^{\Stk})$. Combining this with \Cref{prop:Glo_Mfld_Vs_Glo_Top} and \Cref{prop:Glo_Mfld_Vs_Glo_Stk}, we thus obtain a chain of equivalences
		\[
		\Shv^{\htp}(\SepStk) \iso \PSh(\Glo^{\Stk}) \isol \PSh(\Glo^{\Mfld}) \iso \PSh(\Glo) = \Glo\Spc. \qedhere
		\]
	\end{proof}
	
	In light of \Cref{thm:MainTheorem}, we may assign to every separated stack its \textit{global homotopy type}:
	
	\begin{definition}
		We define the \textit{global homotopy type functor} as the composite
		\[
		\Piinfty^{\Glo}\colon \Shv(\SepStk) \xrightarrow{L_{\htp}} \Shv^{\htp}(\SepStk) \overset{\ref{thm:MainTheorem}}{\simeq} \Glo\Spc.
		\]
	\end{definition}

	Unwinding the definitions, the global homotopy type of a separated stack $\Xx$ is given on objects by homotopy types of mapping stacks, as stated in the introduction: for every compact Lie group $G$, we have
	\[
		\Piinfty^{\Glo}(\Xx)(\bbB G)\;\simeq\;\Piinfty\,\iHom_{\Shv(\Mfld)}(\bbB G,\Xx).
	\]
	Indeed, under \Cref{thm:MainTheorem} the value at $\bbB G$ is the evaluation $\ev_{\bbB G}$, so $\Piinfty^{\Glo}(\Xx)(\bbB G)\simeq(L_{\htp}\Xx)(\bbB G)$. By \Cref{lem:Comparison_LR_Lhtp_On_EuclStk} this agrees with $(L_{\R}\Xx)(\bbB G)\simeq\abs{[n]\mapsto\Xx(\bbB G\times\Deltaalg{n})}$. As $\bbB G$, $\Deltaalg{n}$ and $\Xx$ all lie in $\Shv(\Mfld)\subseteq\Shv(\SepStk)$, we have $\Xx(\bbB G\times\Deltaalg{n})\simeq\iHom_{\Shv(\Mfld)}(\bbB G,\Xx)(\Deltaalg{n})$, and this geometric realization computes $\Piinfty\,\iHom_{\Shv(\Mfld)}(\bbB G,\Xx)$ by \cite[Proposition~5.1.2]{ADH2021differential}.
	
	\begin{remark}
		We immediately obtain definitions for the global homotopy types of topological and differentiable stacks via precomposition with the geometric morphisms
		\[
		\fgt^*\colon \Shv(\Top) \to \Shv(\Mfld) \qquad \text{ and } \qquad w^!\colon \Shv(\Mfld) \hookrightarrow \Shv(\SepStk),
		\]
		where the first one is given by restriction along the forgetful functor $\fgt\colon \Mfld \to \Top$. Because of the equivalences
		\[
		\Hom_{\Shv(\Top)}(\bbB G \times \Deltaalg{n}, \Xx) \simeq \Hom_{\Shv(\Mfld)}(\bbB G \times \Deltaalg{n}, \fgt^* \Xx) \simeq \Hom_{\Shv(\SepStk)}(\bbB G \times \Deltaalg{n}, w^!\fgt^*\Xx),
		\]
		the resulting functor $\Shv(\Top) \to \Glo\Spc$ agrees with the one defined by \cite[Remark~2.9]{gepnerMeier2020equivariant}.
		
		We warn the reader that the resulting global homotopy type functors on $\Shv(\Top)$ and $\Shv(\Mfld)$ do \textit{not} satisfy descent with respect to arbitrary effective epimorphisms.
	\end{remark}
	
	Recall that the space of \textit{$G$-fixed points} $X^G$ of a global space $X$ is simply defined as its value at the global classifying space $B_{\gl} G \in \Glo$. We will now show that the space of $G$-fixed points of the global homotopy type $\Piinfty^{\Glo}(\Xx)$ of a separated stack $\Xx$ is given by the homotopy type of the $G$-isotropy stack $\Xx^G \in \Shv(\Mfld)$:
	
	\begin{proposition}
		For a compact Lie group $G$, there is a commutative square
		\[
		\begin{tikzcd}
			\Shv(\SepStk) \rar{\Piinfty^{\Glo}} \dar[swap]{(-)^G} & \Glo\Spc \dar{(-)^G} \\
			\Shv(\Mfld) \rar{\Piinfty} & \Spc,
		\end{tikzcd}
		\]
		where the bottom functor $\Piinfty$ is the left adjoint of the constant sheaf functor $\Gamma^*\colon \Spc \hookrightarrow \Shv(\Mfld)$.
	\end{proposition}
	\begin{proof}
		This follows from the following commutative diagram:
		\[
		\begin{tikzcd}
			\Shv(\SepStk) \rar{L_{\htp}} \dar[swap]{(-)^G} & \Shv^{\htp}(\SepStk) \dar[swap]{(-)^G} \rar{\simeq} & \PSh(\Glo^{\Stk}) \dar{(-)^G} \rar{\simeq} & \PSh(\Glo) \dar{(-)^G}\\
			\Shv(\Mfld) \rar{L_{\htp}} & \Shv^{\htp}(\Mfld) \rar{\simeq} & \Spc \rar[equal] & \Spc.
		\end{tikzcd}
		\]
		Here the left square commutes by \Cref{prop:HFixedPointsPreservesEverything}, the middle square commutes since by definition both composites are given by evaluation at $L_{\htp}(\bbB G)$, and the right square commutes by inspection.
	\end{proof}
	
	\subsection{Restriction of isotropy groups}\label{subsec:isotropy}
	\Cref{thm:MainTheorem} admits a variant in which we restrict the class of isotropy groups that are allowed. Recall that the \textit{isotropy group} $G_x$ of a separated stack $\Xx$ at a point $x\colon \pt \to \Xx$ is defined as the following pullback in $\SepStk$:
	\[
	\begin{tikzcd}
		G_x \dar \rar \drar[pullback] & \pt \dar{x} \\
		\pt \rar{x} & \Xx.
	\end{tikzcd}
	\]
	The stack $G_x$ is a compact smooth manifold by applying \cite[Theorem~5.4]{MoerdijkMrcun2003Foliations} to a Lie groupoid representing $\Xx$; compactness of $G_x$ follows from separatedness of $\Xx$. It comes with a canonical group structure turning it into a compact Lie group. See also \cite[Lemmas~3.4.2 and 3.4.3]{Cnossen2023PhD} for additional details.
	
	\begin{definition}
		Let $\Ff$ be a collection of compact Lie groups closed under isomorphisms. We denote by $\SepStk_{\Ff} \subseteq \SepStk$ the full subcategory spanned by those separated stacks all of whose isotropy groups lie in $\Ff$, and we write $\EuclStk_{\Ff} := \EuclStk \cap \SepStk_{\Ff}$ for the subcategory of Euclidean stacks of the form $V\quot G$ for some group $G \in \Ff$ and a $G$-representation $V$ all of whose isotropy groups lie in $\Ff$. We let $\Glo_{\Ff} \subseteq \Glo$ denote the full subcategory spanned by objects of the form $B_{\gl}G$ for $G \in \Ff$.
	\end{definition}
	
	Since any open substack induces isomorphisms at the level of isotropy groups, the subcategory $\SepStk_{\Ff}$ is closed under passing to open substacks, and in particular inherits a Grothendieck topology from $\SepStk$.
	
	\begin{corollary}
		\label[corollary]{cor:Main_Result_For_Families}
		For a collection of compact Lie groups $\Ff$, the equivalence from \Cref{thm:MainTheorem} restricts to an equivalence
		\[
		\Shv^{\htp}(\SepStk_{\Ff}) \iso \PSh(\Glo_{\Ff}).
		\]
	\end{corollary}
	\begin{proof}
		A separated stack $\Xx$ is contained in $\SepStk_{\Ff}$ if and only if it can be covered by open substacks contained in $\EuclStk_{\Ff}$, and thus the equivalence $\Shv^{\htp}(\SepStk) \iso \Shv^{\htp}(\EuclStk)$ restricts to an equivalence $\Shv^{\htp}(\SepStk_{\Ff}) \iso \Shv^{\htp}(\EuclStk_{\Ff})$. The subcategory $\Shv^{\htp}(\EuclStk_{\Ff}) \hookrightarrow \Shv^{\htp}(\EuclStk)$ is generated under colimits by the objects $L_{\htp}(V \quot G)$ for $V\quot G \in \EuclStk_{\Ff}$, which agree up to homotopy with the objects $L_{\htp}(\bbB G)$ for $G \in \Ff$. Passing through the equivalence $\Shv^{\htp}(\EuclStk) \iso \PSh(\Glo)$, this subcategory thus corresponds to the subcategory $\PSh(\Glo_{\Ff}) \subseteq \PSh(\Glo^{\Stk})$ generated by the objects $B_{\gl}G$ for $G \in \Ff$, finishing the proof.
	\end{proof}
	
	\section{Cohesion}\label{subsec:cohesion}
	As a consequence of the above results, we will now deduce that the $\infty$-topos $\Shv(\SepStk)$ is \textit{cohesive}, in the sense of \cite[Definition~4.1.8]{schreiber2013differential}. We further show that the subtopos $\Shv(\Orbfld)$ of sheaves on the site of orbifolds is even \textit{relatively cohesive} over the $\infty$-topos of $\Fin$-global spaces.
	
	\subsection{Cohesion for separated stacks}
	Fix a collection $\Ff$ of compact Lie groups. We start by showing that the homotopy invariant sheaves on $\SepStk_{\Ff}$ are closed under colimits:
	
	\begin{proposition}
		\label[proposition]{cor:Homotopy_Invariant_Sheaves_Closed_Under_Colimits}
		The inclusion $\Shv^{\htp}(\SepStk_{\Ff}) \hookrightarrow \Shv(\SepStk_{\Ff})$ preserves colimits.
	\end{proposition}
	\begin{proof}
		We may equivalently show that the inclusion $\Shv^{\htp}(\EuclStk_{\Ff}) \hookrightarrow \Shv(\EuclStk_{\Ff})$ preserves colimits, see \Cref{rmk:Test_Homotopy_Invariance_On_EuclStk}. We may factor this inclusion as follows:
		\[
		\Shv^{\htp}(\EuclStk_{\Ff}) \hookrightarrow \PSh^{\htp}(\EuclStk_{\Ff}) \hookrightarrow \PSh(\EuclStk_{\Ff}) \xrightarrow{L_{\open}} \Shv(\EuclStk_{\Ff}).
		\]
		The first functor is an equivalence by \Cref{prop:Homotopy_Invariant_Presheaf_On_EuclStk_Is_Sheaf}, hence preserves colimits. For the second functor, observe that the homotopy invariant presheaves are closed under colimits as these are computed pointwise. Finally, $L_{\open}$ preserves colimits as it is a left adjoint.
	\end{proof}
	
	\begin{corollary}
		\label[corollary]{cor:Adjunctions_Homotopy_Invariant_Sheaves}
		There exists a pair of adjunctions as follows:
		\[
		\begin{tikzcd}[column sep = large]
			\Shv(\SepStk_{\Ff}) 
			\ar[shift left=8, "L_{\htp}"{description}]{rrr} 
			\ar[rrr,phantom,shift left=4,"{\scalebox{.6}{$\bot$}}"{description}] 
			\ar[hookleftarrow]{rrr}
			\ar[rrr,phantom,shift right=4,"{\scalebox{.6}{$\bot$}}"{description}] 
			\ar[shift right=8]{rrr} &&& \Shv^{\htp}(\SepStk_{\Ff}).
		\end{tikzcd}
		\]
	\end{corollary}
	\begin{proof}
		The fact that the inclusion $\Shv^{\htp}(\SepStk_{\Ff}) \hookrightarrow \Shv(\SepStk_{\Ff})$ admits a left adjoint $L_{\htp}$ is formal and was discussed before. The fact that it admits a right adjoint follows directly from the adjoint functor theorem, since the inclusion preserves colimits by \Cref{cor:Homotopy_Invariant_Sheaves_Closed_Under_Colimits}.
	\end{proof}
	
	\begin{corollary}
		The functor $w^*\colon \Shv(\Mfld) \hookrightarrow \Shv(\SepStk_{\Ff})$ preserves homotopy invariant sheaves.
	\end{corollary}
	\begin{proof}
		Recall that a sheaf on $\Mfld$ is homotopy invariant if and only if it is in the essential image of the constant sheaf functor $\Gamma^*\colon \Spc \hookrightarrow \Shv(\Mfld)$; see for example \cite[Section~3.1]{ADH2021differential} for an exposition. We must therefore show that the composite $w^*\Gamma^*\colon \Spc \hookrightarrow \Shv(\SepStk_{\Ff})$ lands in the subcategory of homotopy invariant sheaves, or equivalently, that the following diagram commutes:
		\[
		\begin{tikzcd}
			\Shv(\SepStk_{\Ff}) \rar[hookleftarrow] & \Shv^{\htp}(\SepStk_{\Ff})\\
			\Shv(\Mfld) \uar[hookrightarrow]{w^*} \rar[hookleftarrow]{\Gamma^*} & \Spc, \uar[hookrightarrow]
		\end{tikzcd}
		\]
		where the right vertical map is the unique left exact left adjoint from $\Spc$ to the $\infty$-topos $\Shv^{\htp}(\SepStk_{\Ff})$. But this is clear from $\Spc$ being the free cocomplete $\infty$-category on a point, since both composites preserve colimits and preserve the terminal object.
	\end{proof}
	
	\begin{theorem}
		\label[theorem]{thm:Cohesion}
		The $\infty$-topos $\Shv(\SepStk_{\Ff})$ is \emph{cohesive}, in the sense that the following two conditions are satisfied:
		\begin{enumerate}[(1)]
			\item The constant sheaf functor $\Gamma^*\colon \Spc \to \Shv(\SepStk_{\Ff})$ is fully faithful and admits a finite-product-preserving left adjoint $\Piinfty = \Gamma_{\sharp}\colon \Shv(\SepStk_{\Ff}) \to \Spc$;
			\item The global sections functor $\Gamma_*\colon \Shv(\SepStk_{\Ff}) \to \Spc$ admits a fully faithful right adjoint $\Gamma^! \colon \Spc \hookrightarrow \Shv(\SepStk_{\Ff})$.
		\end{enumerate}
	\end{theorem}
	\begin{proof}
		For (1), note that $\Gamma^*$ is fully faithful as it is a composite of $\Spc \hookrightarrow \Shv(\Mfld)$ and $w^*\colon \Shv(\Mfld) \hookrightarrow \Shv(\SepStk)$. Since it alternatively factors as
		\[
		\Spc \xrightarrow{\mathrm{const}} \PSh(\Glo_{\Ff}) \simeq \Shv^{\htp}(\SepStk_{\Ff}) \hookrightarrow \Shv(\SepStk_{\Ff}),
		\]
		it admits a left adjoint given by the composite
		\[
		\Shv(\SepStk_{\Ff}) \xrightarrow{L_{\htp}} \Shv^{\htp}(\SepStk_{\Ff}) \simeq \PSh(\Glo_{\Ff}) \xrightarrow{L} \Spc,
		\]
		where the last functor $L$ is the left Kan extension of the constant functor $\const_*\colon \Glo^{\Stk}_{\Ff}  \to \Spc$. Since $\const_*$ preserves finite products, so does $L$. The functor $L_{\htp}$ also preserves finite products as a consequence of \cite[Proposition~3.4]{hoyois2017sixoperations}, see also \cite[Corollary~4.2.14]{Cnossen2023PhD}, and so the composite is finite-product-preserving.
		
		Part (2) is immediate from the adjunctions from \Cref{prop:Adjunctions_SepStk_and_Mfld} and the fact that $\Gamma_*\colon \Shv(\Mfld) \to \Spc$ admits a fully faithful right adjoint $\Gamma^!\colon \Spc \hookrightarrow \Shv(\Mfld)$.
	\end{proof}
	
	\subsection{Cohesion for orbifolds}	\label{subsec:Cohesion_for_orbifolds}
	The situation becomes even better if we restrict attention to orbifolds.
	
	\begin{definition}[{Orbifold, cf.\ \cite{moerdijk2002orbifolds}}]
		A separated stack $\Xx$ is called an \textit{orbifold} if all its isotropy groups are finite groups. We let $\Orbfld \subseteq \SepStk$ denote the full subcategory spanned by the orbifolds. More generally, given a collection $\Ff$ of finite groups, we will write $\Orbfld_{\Ff} := \SepStk_{\Ff} \subseteq \SepStk$ for the subcategory of separated stacks with isotropy in $\Ff$.
	\end{definition}
	
	The relevant fact about finite groups is the fact that the classifying sheaf $\bbB G \in \Shv(\Orbfld_{\Ff})$ is already homotopy invariant for every finite group $G$, see \Cref{cor:BG_Homotopy_Invariant}. In the remainder of this subsection we consider $\Glo_{\Ff}$ as a subcategory of $\Shv(\Orbfld_{\Ff})$ via the composite $\Glo_{\Ff}\simeq \Glo_{\Ff}^{\Stk} \subset \Shv(\Orbfld_{\Ff})$.
	
	\begin{lemma}
		If $\Ff$ consists of finite groups, the subcategory $\Glo_{\Ff} \subseteq \Shv(\Orbfld_{\Ff})$ agrees with the full subcategory of $\Orbfld_{\Ff} \subseteq \Shv(\Orbfld_{\Ff})$ spanned by the representable objects $\bbB G \in \Orbfld_{\Ff}$ for $G \in \Ff$.
	\end{lemma}
	\begin{proof}
		By definition, the image of $\Glo_{\Ff}$ in $\Shv(\Orbfld_{\Ff})$ is spanned by the sheaves $L_{\htp}(\bbB G)$ for $G \in \Ff$, but this agrees with $\bbB G$ since it is already a homotopy invariant sheaf by \Cref{cor:BG_Homotopy_Invariant}.
	\end{proof}
	
	\begin{lemma}
		The inclusion functor $i\colon \Glo_{\Ff} \hookrightarrow \Orbfld_{\Ff}$ is both a continuous and cocontinuous morphism of sites, where we equip $\Glo_{\Ff}$ with the trivial Grothendieck topology. In particular, we obtain adjunctions
		\begin{equation*}
			\begin{tikzcd}[column sep = large]
				\Shv(\Orbfld_{\Ff})
				\ar[rr,"{ i_* }"{description}]
				&&
				\Shv(\Glo_{\Ff}) = \PSh(\Glo_{\Ff}),
				\ar[ll,hook',shift right=18pt,"{ i^* }"{description}]
				\ar[ll,hook',shift right=-18pt,"{ i^! }"{description}]
				\ar[ll,phantom,shift right=9pt,"{\scalebox{.6}{$\bot$}}"]
				\ar[ll,phantom,shift right=-9pt,"{\scalebox{.6}{$\bot$}}"]
			\end{tikzcd}
		\end{equation*}
	\end{lemma}
	\begin{proof}
		As the Grothendieck topology on $\Glo_{\Ff}$ is trivial, the inclusion is automatically continuous. For cocontinuity, it suffices to observe that the only open cover of the orbifold $\bbB G$ is the trivial open cover: open covers of a stack are in one-to-one correspondence with open covers of the coarse moduli space of the stack, which in case of $\bbB G$ is the one-point space.
	\end{proof}
	
	\begin{lemma}
		The inclusion $i^*\colon \PSh(\Glo_{\Ff}) \hookrightarrow \Shv(\Orbfld_{\Ff})$ agrees with the composite
		\[
		\PSh(\Glo_{\Ff}) \iso \Shv^{\htp}(\Orbfld_{\Ff}) \hookrightarrow \Shv(\Orbfld_{\Ff}),
		\]
		where the first functor is the (inverse of the) equivalence from \Cref{cor:Main_Result_For_Families} and the second functor is the inclusion. In particular, this functor admits a further left adjoint $i_!$ which preserves finite products.
	\end{lemma}
	\begin{proof}
		Both of these functors preserve colimits and they agree on the representable objects, hence agree on all of $\PSh(\Glo_{\Ff})$. The last statement is an immediate consequence of the fact that the inclusion of homotopy invariant sheaves into all sheaves admits a left adjoint $L_{\htp}$ which preserves finite products.
	\end{proof}
	
	\begin{corollary}
		The functor $i_*\colon \Shv(\Orbfld_{\Ff}) \to \PSh(\Glo_{\Ff})$ exhibits its source as a cohesive $\infty$-topos over its target. \qed
	\end{corollary}
	
	All in all, we have obtained the following adjunctions:
	\[
	\begin{tikzcd}[column sep = large]
		\Shv(\Mfld) \,\,
		\ar[hookrightarrow, shift left=8, "w^*"{description}]{rrr} 
		\ar[rrr,phantom,shift left=4,"{\scalebox{.6}{$\bot$}}"{description}] 
		\ar["w_*"{description}, leftarrow]{rrr}
		\ar[rrr,phantom,shift right=4,"{\scalebox{.6}{$\bot$}}"{description}] 
		\ar[hookrightarrow,shift right=8, "w^!"{description}]{rrr}
		&&& 
		\,\,\Shv(\Orbfld)\,\,
		\ar[rrr,shift right=9pt,"{ i_{*} }"{description}]
		\ar[rrr,shift left=27pt,"{ i_{!} }"{description}]
		&&& \,\,\Glo\Spc.
		\ar[lll,hook',shift right=9pt,"{ i^* }"{description}]
		\ar[lll,hook',shift right=-27pt,"{ i^! }"{description}]
		\ar[lll,phantom,shift right=0pt,"{\scalebox{.6}{$\bot$}}"]
		\ar[lll,phantom,shift right=-18pt,"{\scalebox{.6}{$\bot$}}"]
		\ar[lll,phantom,shift right=18pt,"{\scalebox{.6}{$\bot$}}"]
	\end{tikzcd}
	\]
	
	To end the section, we will give a comparison between our cohesive $\infty$-topos $\Shv(\Orbfld)$ and the ``singular-cohesive $\infty$-topos'' $\bH$ defined by Sati and Schreiber \cite[Definition~3.48]{satiSchreiber2020proper}.
	
	\begin{definition}
		The \textit{singular-cohesive $\infty$-topos} $\bH$ is defined as
		\[
		\bH := \Fun((\Glo_{\Fin})\catop, \Shv(\Eucl)).
		\]
		Observe that $\bH$ is equivalent to the sheaf $\infty$-topos $\Shv(\Glo_{\Fin} \times \Eucl)$, where $\Glo_{\Fin} \times \Eucl$ denotes the product of sites and where $\Glo_{\Fin}$ is equipped with the trivial Grothendieck topology as before.
	\end{definition}
	
	\begin{lemma}
		\label[lemma]{lem:Sati_Schreiber_Fully_Faithful}
		The functor $\Glo_{\Fin} \times \Eucl \to \Orbfld$ given by $(\bbB G, \R^n) \mapsto \bbB G \times \R^n$ is fully faithful.
	\end{lemma}
	\begin{proof}
		Given objects $(\bbB G, \R^n)$ and $(\bbB H, \R^m)$, we may equivalently show that the maps
		\[
		\Hom_{\Orbfld}(\bbB G, \bbB H) \to \Hom_{\Orbfld}(\bbB G \times \R^n, \bbB H)
		\]
		and
		\[
		\Hom_{\Orbfld}(\R^n, \R^m) \to \Hom_{\Orbfld}(\bbB G \times \R^n, \R^m)
		\]
		are equivalences. For the first map, this is immediate from homotopy invariance of the stack $\bbB H$, see \Cref{cor:BG_Homotopy_Invariant}. For the second map, we may write $\bbB G$ as the colimit in $\Orbfld$ of the simplicial nerve of $G$ to obtain
		\[
		\Hom_{\Orbfld}(\bbB G \times \R^n, \R^m) \iso \mathrm{eq}\left(\Hom_{\Orbfld}(\R^n, \R^m) \rightrightarrows \Hom_{\Orbfld}(\R^n \times G, \R^m)\right).
		\]
		Since this is an equalizer of sets of two identical maps, this equalizer agrees with $\Hom_{\Orbfld}(\R^n, \R^m)$, finishing the proof.
	\end{proof}
	
	\begin{proposition}
		\label[proposition]{prop:Comparison_Sati_Schreiber}
		The inclusion $\Glo_{\Fin} \times \Eucl \hookrightarrow \Orbfld$ is both a continuous and cocontinuous morphism of sites. In particular, restriction and left/right Kan extension along this inclusion define adjunctions
		\begin{equation*}
			\begin{tikzcd}[column sep = large]
				\bH \simeq \Shv(\Glo_{\Fin} \times \Eucl)
				\ar[rr,hook,shift right=18pt]
				\ar[rr,hook,shift right=-18pt]
				\ar[rr,phantom,shift right=9pt,"{\scalebox{.6}{$\bot$}}"]
				\ar[rr,phantom,shift right=-9pt,"{\scalebox{.6}{$\bot$}}"]
				&&
				\Shv(\Orbfld).
				\ar[ll]
			\end{tikzcd}
		\end{equation*}
	\end{proposition}
	\begin{proof}
		It is clear that the inclusion preserves open coverings, and hence is continuous. For cocontinuity, note that any open cover of the orbifold $\bbB G \times \R^n$ consists of orbifolds of the form $\bbB G \times U$ for open subsets $U \subseteq \R^n$, which may be refined to an open cover consisting of objects in $\Glo_{\Fin} \times \Eucl$ by covering the subsets $U$ by Euclidean subspaces.
	\end{proof}
	
We conclude by fleshing out the relationship between $\bH$ and $\Shv(\Orbfld)$ exhibited in the preceding proposition. The $\infty$-topos $\bH$ is a product $\infty$-topos (in the category of $\infty$-topoi and geometric morphisms, or, equivalently, Lurie's tensor product in $\PrL$)
$$
\bH \simeq \Shv(\Mfld) \otimes \Glo\Spc_{\Fin},
$$
and the functor $\Shv(\Orbfld)\to\bH$ of Proposition~\ref{prop:Comparison_Sati_Schreiber} is the constituent pushforward of a geometric morphism into this product. It is therefore natural to ask for its components, which are given by the previously introduced functors
$$
w_* :  \Shv(\Orbfld)\to\Shv(\Mfld)
\qquad\text{and}\qquad
i_* :  \Shv(\Orbfld)\to\Glo\Spc_{\Fin}.
$$
		
The essential image of the inclusion $\Glo_{\Fin} \times \Eucl \hookrightarrow \Orbfld$ from \Cref{lem:Sati_Schreiber_Fully_Faithful} is spanned by those Euclidean stacks $\R^n\quot G$ for which the $G$-action on $\R^n$ is trivial, so that we obtain the inclusions of sites $\Mfld \subseteq \Glo_{\Fin} \times \Eucl  \subseteq \Orbfld$. The three $\infty$-toposes $\Shv(\Mfld) \subseteq \bH \subseteq \Shv(\Orbfld)$ all contain $\Orbfld$ as a full subcategory, and (similar to the discussion in the introduction for $\SepStk$), may be viewed as sheaves on $\Orbfld$, with different topologies. These topologies essentially characterise which types of orbifold singularities are taken as primitive (in a somewhat different direction to varying $\Ff$).  

Like $\Shv(\Orbfld)$, the $\infty$-topos $\bH$ is cohesive over $\Glo\Spc_{\Fin}$, and the restriction of the functor $i_!: \Shv(\Orbfld) \to \PSh(\Glo_{\Fin})$ to $\Shv(\Mfld)$ coincides with the restriction of the corresponding left adjoint $\operatorname{Shp}:  \bH \to \PSh(\Glo_{\Fin})$ to $\Shv(\Mfld)$. The $\infty$-topos $\bH$ is moreover cohesive over $\Shv(\Mfld)$, i.e., the functor $\Shv(\Mfld) \to \bH$ corresponding to $w^*$, admits a finite product preserving left adjoint. This functor sends any orbifold to its $0$-truncation. While this functor is well-defined in our setting as well, we do not expect it to form a left adjoint to $w^*$ (in fact, we have no reason to believe that $w^*$ admits a left adjoint at all). 
	
	\printbibliography[heading=bibintoc]
	
	\Addresses
	
\end{document}